\documentclass[10pt]{article}
\usepackage{cite}
\usepackage[utf8]{inputenc}
\usepackage{geometry}
\usepackage{graphicx}
\usepackage{booktabs}
\usepackage{array}
\usepackage{paralist}
\usepackage{verbatim}
\usepackage{subfig}
\usepackage{url}
\usepackage{amsmath,amsthm,amssymb,enumerate}
\usepackage{fancyhdr}
\usepackage{sectsty}
\usepackage[nottoc,notlof,notlot]{tocbibind}
\usepackage[titles,subfigure]{tocloft}

\newtheorem{thm}{Theorem}[section]
\newtheorem{cor}[thm]{Corollary}
\newtheorem{lem}[thm]{Lemma}
\newtheorem{prop}[thm]{Proposition}

\theoremstyle{remark}

\theoremstyle{definition}
\newtheorem{define}[thm]{Definition}

\title{The Theory of Pseudo-differential Operators
on the Noncommutative $n$-Torus}
\author{Jim Tao}
\date{\today}
\begin{document}
\maketitle

\section{Introduction}
The methods of spectral geometry are useful for investigating
the metric aspects of noncommutative geometry
\cite{ncdg,ncg,cmindex,ncgqfm} and in these contexts require extensive use
of pseudo-differential operators. In the foundational paper \cite{ccstar},
Connes showed that, by direct analogy with the theory
\cite{gilkey,raymond,wong} of pseudo-differential operators on $\mathbb R^n$,
one may derive a similar pseudo-differential calculus on noncommutative $n$
tori $\mathbb T_{\theta}^n$. With the development of this calculus
came many results concerning the local differential geometry of noncommutative
tori for $n=2,4$, as shown in the groundbreaking paper \cite{ct}
in which the Gauss--Bonnet theorem on $\mathbb T_{\theta}^2$ is proved
and later papers \cite{cmcurvature,fkgb,fk2t,fk4t,fkwc}.
In these papers, the flat geometry of $\mathbb T_{\theta}^n$ which was studied
in \cite{ccstar} is conformally perturbed using a Weyl factor given
by a positive invertible smooth element in $C^{\infty}(\mathbb T_{\theta}^n)$.
Connes' pseudodifferential calculus is critically used to apply heat
kernel techniques to geometric operators on $\mathbb T_{\theta}^n$ to
derive small time heat kernel expansions that encode local geometric
information such as scalar curvature. As discovered in
\cite{cmcurvature, fk2t, fk4t}, a purely noncommutative feature that
appears in the computations and in the final formula for the curvature
is the modular automorphism of the state implementing the conformal
perturbation of the metric.

Certain details of the proofs in the derivation of the calculus in
\cite{ccstar} were omitted, such as the evaluation of oscillatory integrals,
so we make it the objective of this paper to fill in all the details. After
reproving in more detail the formula for the symbol of the adjoint of a
pseudo-differential operator and the formula for the symbol of a product of
two pseudo-differential operators, we define the corresponding analog of
Sobolev spaces for which we prove the Sobolev and Rellich lemmas. We then
extend these results to finitely generated projective right modules over the
noncommutative $n$ torus.

We list these results below.
\begin{thm}
Suppose $P$ is a pseudodifferential operator with symbol
$\sigma(P)=\rho=\rho(\xi)$ of order $M$.
Then the symbol of the adjoint $P^*$ is of order $M$ and satisfies
$$\sigma(P^*)\sim\sum_{\ell\in\mathbb Z_{\ge 0}^n}
\frac{\partial^{\ell}\delta^{\ell}[(\rho(\xi))^*]}
{\ell_1!\ell_2!}.$$
\end{thm}
\begin{thm}
Suppose that $P$ is a pseudodifferential operator with symbol
$\sigma(P)=\rho=\rho(\xi)$ of order $M_1$, and $Q$ is a pseudodifferential
operator with symbol $\sigma(Q)=\phi=\phi(\xi)$ of order $M_2$.
Then the symbol of the product $QP$ is of order $M_1+M_2$ and satisfies
$$\sigma(QP)\sim\sum_{\ell\in\mathbb Z_{\ge 0}^n}
\frac{1}{\ell_1!\ell_2!}\partial^{\ell}\phi(\xi)\delta^{\ell}\rho(\xi),$$
where $\partial^{\ell}:=\prod_j\partial_j^{\ell_j}$
and $\delta^{\ell}:=\prod_j\delta_j^{\ell_j}$.
\end{thm}
\begin{thm}
For $s>k+1$, $H^s\subseteq A_{\theta}^k$.
\end{thm}
\begin{thm}
Let $\{a_N\}\in A_{\theta}^{\infty}$ be a sequence. Suppose that there
is a constant $C$ so that $||a_N||_s\le C$ for all $N$. Let $s>t$. Then
there is a subsequence $\{a_{N_j}\}$ that converges in $H^t$.
\end{thm}
\begin{thm}
\begin{enumerate}[(a)]
\item For a pseudodifferential operator $P$ with $r\times r$ matrix valued
symbol $\sigma(P)=\rho=\rho(\xi)$, the
symbol of the adjoint $P^*$ satisfies
$$\sigma(P^*)\sim\sum_{(\ell_1,\ldots,\ell_n)\in(\mathbb Z_{\ge 0})^n}
\frac{\partial_1^{\ell_1}\cdots\partial_n^{\ell_n}
\delta_1^{\ell_1}\cdots\delta_n^{\ell_n}(\rho(\xi))^*}{\ell_1!\cdots\ell_n!}.$$
\item If $Q$ is a pseudodifferential operator with $r\times r$ matrix valued symbol
$\sigma(Q)=\rho'=\rho'(\xi)$,
then the product $PQ$ is also a pseudodifferential operator and has symbol
$$\sigma(PQ)\sim\sum_{(\ell_1,\ldots,\ell_n)\in(\mathbb Z_{\ge 0})^n}
\frac{\partial_1^{\ell_1}\cdots\partial_n^{\ell_n}(\rho(\xi))
\delta_1^{\ell_1}\cdots\delta_n^{\ell_n}(\rho'(\xi))}{\ell_1!\cdots\ell_n!}.$$
\end{enumerate}
\end{thm}
\begin{thm}
For $s>k+1$, $H^s\subseteq(A_{\theta}^k)^re$.
\end{thm}
\begin{thm}
Let $\{\vec{a}_N\}\in(A_{\theta}^{\infty})^re$ be a sequence. Suppose that there
is a constant $C$ so that $||\vec{a}_N||_s\le C$ for all $N$. Let $s>t$. Then
there is a subsequence $\{\vec{a}_{N_j}\}$ that converges in $H^t$.
\end{thm}

\section{Preliminaries}
Fix some skew symmetric $n\times n$ matrix $\theta$ with upper
triangular entries in $\mathbb R\backslash\mathbb Q$ that are linearly
independent over $\mathbb Q$. Consider the irrational rotation
$C^*$-algebra $A_{\theta}$ with $n$ unitary generators $U_1,\ldots,U_n$ which
satisfy $U_kU_j=e^{2\pi i\theta_{j,k}}U_jU_k$ and $U_j^*=U_j^{-1}$. Let
$\{\alpha_s\}_{s\in\mathbb R^n}$ be a $n$-parameter group of automorphisms
given by $\prod_jU_j^{m_j}\mapsto e^{is\cdot m}\prod_jU_j^{m_j}$. We define
the subalgebra $A_{\theta}^k$ of $C^k$ elements of $A_{\theta}$ to be those
$a\in A_{\theta}$ such that the mapping $\mathbb R^n\rightarrow A_{\theta}$
given by $s\mapsto\alpha_s(a)$ is $C^k$, and we define the subalgebra
$A_{\theta}^{\infty}$ of smooth elements of $A_{\theta}$ to be those
$a\in A_{\theta}$ such that the mapping $\mathbb R^n\rightarrow A_{\theta}$
given by $s\mapsto\alpha_s(a)$ is smooth. An alternative definition of the
subalgebra $A_{\theta}^{\infty}$ of smooth elements is the elements in
$A_{\theta}$ that can be expressed by an expansion of the form
$\sum_{m\in\mathbb Z^n}a_m\prod_jU_j^{m_j}$, where the sequence
$\{a_m\}_{m\in\mathbb Z^n}$ is in the Schwartz space $\mathcal S(\mathbb Z^n)$
in the sense that, for all $\alpha\in\mathbb Z^n$,
$$\sup_{m\in\mathbb Z^n}(\prod_j|m_j|^{\alpha_j}|a_m|)<\infty.$$
Define the trace $\tau:A_{\theta}\rightarrow\mathbb C$ by
$\tau(\prod_jU_j^{m_j})=0$ for $m_j$ not all zero
and $\tau(1)=1$ and define an inner product
$\langle\cdot,\cdot\rangle:A_{\theta}\times A_{\theta}\rightarrow\mathbb C$
by $\langle a,b\rangle=\tau(b^*a)$ with induced norm
$||\cdot||:A_{\theta}\rightarrow\mathbb R_{\ge 0}$.
Let $D_j=-i\partial_j$ and
define derivations $\delta_j$ by the relations $\delta_j(U_j)=U_j$ and
$\delta_j(U_k)=0$ for $j\ne k$. For convenience, denote
$\partial^{\ell}:=\prod_j\partial_j^{\ell_j}$,
$\delta^{\ell}:=\prod_j\delta_j^{\ell_j}$, and
$D^{\ell}:=\prod_jD_j^{\ell_j}$.
We define a map $\psi:\rho\mapsto P_{\rho}$ assigning a pseudo-differential
operator on $A_{\theta}^{\infty}$ to a symbol
$\rho\in C^{\infty}(\mathbb R^n,A_{\theta}^{\infty})$.
\begin{define}
For $\rho\in C^{\infty}(\mathbb R^n,A_{\theta}^{\infty})$,
let $P_{\rho}$ be the pseudo-differential operator
sending arbitrary $a\in A_{\theta}^{\infty}$ to
$$P_{\rho}(a):=\frac{1}{(2\pi)^n}
\int_{\mathbb R^n}\!\int_{\mathbb R^n}\!
e^{-is\cdot\xi}\rho(\xi)\alpha_s(a)\,\mathrm ds\,\mathrm d\xi.$$
\end{define}
The integral above does not converge absolutely; it is an oscillatory integral.
We define oscillatory integrals below as in \cite{raymond}.
\begin{define}
Let $q$ be a nondegenerate real quadratic form on $\mathbb R^n$,
$a$ be a $C^{\infty}$ complex-valued function
defined on $\mathbb R^n$ such that the functions
$(1+|x|^2)^{-m/2}\partial^{\alpha}a(x)$ are bounded on $\mathbb R^n$
for all $\alpha\in\mathbb Z_{\ge 0}^n$, and $\varphi$ be a Schwartz function,
i.e. the functions $x^{\alpha}\partial^{\beta}\varphi(x)$ are bounded on
$\mathbb R^n$ for all pairs $\alpha,\beta\in\mathbb Z_{\ge 0}^n$. Suppose
further that $\varphi(0)=1$. Then the limit
$$\lim_{\epsilon\rightarrow 0}
\int\!e^{iq(x)}a(x)\varphi(\epsilon x)\,\mathrm dx$$
exists, is independent of $\varphi$ (as long as $\varphi(0)=1$), and is
equal to $\int\!e^{iq(x)}a(x)\,\mathrm dx$ when $a\in L^1$. When $a\not\in L^1$,
we continue to denote this limit by $\int\!e^{iq(x)}a(x)\,\mathrm dx$,
and have an estimate
$$\left|\int\!e^{iq(x)}a(x)\,\mathrm dx\right|
\le C_{q,m}\max_{|\alpha|\le m+n+1}\inf\{U\in\mathbb R:
|(1+|x|^2)^{-m/2}\partial^{\alpha}a|\le U\text{ almost everywhere}\}$$
where $C_{q,m}$ depends only on the quadratic form $q$ and the order $m$.
\end{define}
As shown in \cite{raymond}, oscillatory integrals behave essentially like
absolutely convergent integrals in that one can still make changes of
variables, integrate by parts, differentiate under the integral sign, and
interverse integral signs. Given certain conditions on $\rho$, $P_{\rho}(a)$
satisfies the conditions of the oscillatory integral, and we can evaluate
$P_{\rho}(a)$.
\begin{define}
An element $\rho=\rho(\xi)=\rho(\xi_1,\ldots,\xi_n)$ of
$C^{\infty}(\mathbb R^n,A_{\theta}^{\infty})$ is a symbol of order $m$ if and
only if for all non-negative integers $i_1,\ldots,i_n,j_1,\ldots,j_n$
$$||\delta_1^{i_1}\cdots\delta_n^{i_n}(\partial_1^{j_1}\cdots\partial_n^{j_n}
\rho)(\xi)||\le C_{\rho}(1+|\xi|)^{m-|j|}.$$
\end{define}
Example 2.6(i) of \cite{raymond} gives a convenient formula for evaluating
the oscillatory integrals that appear in our calculation of $P_{\rho}(a)$.
\begin{prop}\label{osc}
Suppose that, for some $m$, $a$ is a $C^{\infty}$ complex-valued function
defined on $\mathbb R^n$ such that the functions
$(1+|x|^2)^{-m/2}\partial^{\alpha}a(x)$ are bounded on $\mathbb R^n$
for all $\alpha\in\mathbb Z_{\ge 0}^n$.
Then
$$\frac{1}{(2\pi)^n}\int_{\mathbb R^n}\!\int_{\mathbb R^n}\!
e^{-iy\cdot\eta}a(y)\,\mathrm dy\,\mathrm d\eta
=\frac{1}{(2\pi)^n}\int_{\mathbb R^n}\!\int_{\mathbb R^n}\!
e^{-iy\cdot\eta}a(\eta)\,\mathrm dy\,\mathrm d\eta
=a(0).$$
\end{prop}
We apply Proposition \ref{osc} to get a basic result.
\begin{lem}\label{opn}
Let $a=\sum_{m\in\mathbb Z^n}a_m\prod_jU_j^{m_j}$
be an arbitrary element of $A_{\theta}^{\infty}$ and let
$\rho\in C^{\infty}(\mathbb R^n,A_{\theta}^{\infty})$ be a symbol of order $M$.
Then $$P_{\rho}(a)=\sum_{m\in\mathbb Z^n}\rho(m)a_m\prod_{j=1}^nU_j^{m_j}.$$
\end{lem}
\begin{proof}
First consider the case $a=\prod_jU_j^{m_j}$. We get
\begin{align*}
P_{\rho}\left(\prod_{j=1}^nU_j^{m_j}\right)
&=\frac{1}{(2\pi)^n}\int_{\mathbb R^n}\!\int_{\mathbb R^n}\!
e^{-is\cdot\xi}\rho(\xi)\alpha_s\left(\prod_{j=1}^nU_j^{m_j}\right)
\,\mathrm ds\,\mathrm d\xi \\
&=\frac{1}{(2\pi)^n}\int_{\mathbb R^n}\!\int_{\mathbb R^n}\!
e^{-is\cdot\xi}\rho(\xi)e^{is\cdot m}\prod_{j=1}^nU_j^{m_j}
\,\mathrm ds\,\mathrm d\xi \\
&=\frac{1}{(2\pi)^n}\int_{\mathbb R^n}\!\int_{\mathbb R^n}\!
e^{-is\cdot(\xi-m)}\rho(\xi)\,\mathrm ds\,\mathrm d\xi\,
\prod_{j=1}^nU_j^{m_j} \\
&=\frac{1}{(2\pi)^n}\int_{\mathbb R^n}\!\int_{\mathbb R^n}\!
e^{-is\cdot\eta}\rho(\eta+m)\,\mathrm ds\,\mathrm d\eta
\,\prod_{j=1}^nU_j^{m_j} \\
&=\rho(m)\prod_{j=1}^nU_j^{m_j},
\end{align*}
as desired, having substituted $\eta=\xi-m$ and applied the result of
Proposition \ref{osc}.
Now consider the general case
$a=\sum_{m\in\mathbb Z^n}a_m\prod_jU_j^{m_j}$.
Since $\alpha_s$ is an automorphism on $A_{\theta}$, we get
$$P_{\rho}(a)=\sum_{m\in\mathbb Z^n}\rho(m)a_m\prod_{j=1}^nU_j^{m_j},$$
and we are done.
\end{proof}
\section{Asymptotic formula for the symbol of the adjoint
of a pseudo-differential operator}
Here we prove the formula for the symbol of the adjoint for the noncommutative
$n$ torus, adapting the proof of Lemma 1.2.3 of \cite{gilkey} to the
noncommutative $n$ torus.
\begin{thm}\label{adj}
Suppose $P$ is a pseudodifferential operator with symbol
$\sigma(P)=\rho=\rho(\xi)$ of order $M$.
Then the symbol of the adjoint $P^*$ is of order $M$ and satisfies
$$\sigma(P^*)\sim\sum_{\ell\in\mathbb Z_{\ge 0}^n}
\frac{\partial^{\ell}\delta^{\ell}[(\rho(\xi))^*]}
{\ell_1!\ell_2!}.$$

\end{thm}
\begin{proof}
Let $a,b\in A_{\theta}^{\infty}$.
We have
\begin{align*}
\langle P_{\rho}(a),b\rangle
&= \tau\left(b^*\frac{1}{(2\pi)^n}\int_{\mathbb R^n}\!\int_{\mathbb R^n}\!
e^{-is\cdot\xi}\rho(\xi)\alpha_s(a)\,\mathrm ds\,\mathrm d\xi\right) \\
&= \frac{1}{(2\pi)^n}\int_{\mathbb R^n}\!\int_{\mathbb R^n}\!
e^{-is\cdot\xi}\tau(b^*\rho(\xi)\alpha_s(a))\,\mathrm ds\,\mathrm d\xi \\
&= \frac{1}{(2\pi)^n}\int_{\mathbb R^n}\!\int_{\mathbb R^n}\!e^{-is\cdot\xi}
\tau(\alpha_{-s}(\rho(\xi)^*b)^*a)\,\mathrm ds\,\mathrm d\xi \\
&= \tau\left(\left(\frac{1}{(2\pi)^n}\int_{\mathbb R^n}\!
\int_{\mathbb R^n}\!e^{+is\cdot\xi}
\alpha_{-s}(\rho(\xi)^*b)\,\mathrm ds\,\mathrm d\xi\right)^*a\right) \\
&= \langle a,P_{\rho}^*(b)\rangle
\end{align*}
where
\begin{align*}
P_{\rho}^*(b) &= \frac{1}{(2\pi)^n}\int_{\mathbb R^n}\!\int_{\mathbb R^n}\!
e^{+is\cdot\xi}\alpha_{-s}(\rho(\xi)^*b)\,\mathrm ds\,\mathrm d\xi \\
&= \frac{1}{(2\pi)^n}\int_{\mathbb R^n}\!\int_{\mathbb R^n}\!e^{+is\cdot\xi}
\alpha_{-s}(\rho(\xi)^*)\alpha_{-s}(b)\,\mathrm ds\,\mathrm d\xi \\
&= \sum_{m,k}\frac{1}{(2\pi)^n}\int_{\mathbb R^n}\int_{\mathbb R^n}
e^{+is\cdot\xi}e^{+is\cdot m}\rho_m(\xi)^*e^{-is\cdot k}b_k
\,\mathrm ds\,\mathrm d\xi\,U_n^{-m_n}\cdots U_1^{-m_1}
U_1^{k_1}\cdots U_n^{k_n} \\
&= \sum_{m,k}\frac{1}{(2\pi)^n}\int_{\mathbb R^n}\int_{\mathbb R^n}
e^{-is\cdot\eta}\rho_m((k-m)-\eta)^*b_k
\,\mathrm ds\,\mathrm d\eta\,U_n^{-m_n}\cdots U_1^{-m_1}
U_1^{k_1}\cdots U_n^{k_n} \\
&= \sum_{m,k}\rho_m(k-m)^*b_k
U_n^{-m_n}\cdots U_1^{-m_1}U_1^{k_1}\cdots U_n^{k_n} \\
&= \sum_{m,k}(\rho_m(k-m)U_1^{m_1}\cdots U_n^{m_n})^*
(b_kU_1^{k_1}\cdots U_n^{k_n}),
\end{align*}
so 
\begin{align*}
\sigma(P_{\rho}^*)(\xi)
&= \left[\sum_m\rho_m(\xi-m)\prod_{j=1}^nU_j^{m_j}\right]^* \\
&= \left[\frac{1}{(2\pi)^n}\int_{\mathbb R^n}\!\int_{\mathbb R^n}\!
e^{-ix\cdot y}\sum_m\rho_m(\xi-y)\alpha_x\left(\prod_{j=1}^nU_j^{m_j}\right)
\,\mathrm dx\,\mathrm dy\right]^* \\
&= \left[\frac{1}{(2\pi)^n}\int_{\mathbb R^n}\!\int_{\mathbb R^n}\!
e^{-ix\cdot y}\alpha_x(\rho(\xi-y))\,\mathrm dx\,\mathrm dy\right]^*
\end{align*}
We have
\begin{align*}
\rho(\xi-y) &= \sum_{|\ell|<N_1}\frac{(-y)^{\ell}}{\ell!}
(\partial^{\ell}\rho)(\xi)+R_{N_1}(\xi,y) \\
\alpha_x(\rho(\xi-y)) &= \sum_{|\ell|<N_1}\sum_m
\frac{(-y)^{\ell}}{\ell!}(\partial^{\ell}\rho_m)(\xi)
e^{ix\cdot m}\left(\prod_{j=1}^nU_j^{m_j}\right)+\alpha_x(R_{N_1}(\xi,y))
\end{align*}
where
$$R_{N_1}(\xi,y)=N_1\sum_{|\ell|=N_1}\frac{(-y)^{\ell}}{\ell!}
\int_0^1\!(1-\gamma)^{N_1-1}(\partial^{\ell}\rho)(\xi-y\gamma)
\,\mathrm d\gamma$$
so the corresponding symbol is
\begin{align*}
\sigma(P_{\rho}^*)(\xi)
&= \left[\sum_{|\ell|<N_1}\sum_m\frac{(-m)^{\ell}}
{\ell!}(\partial^{\ell}\rho_m)(\xi)\prod_{j=1}^mU_j^{m_j}
+T_{N_1}(\xi,y)\right]^* \\
&= \sum_{|\ell|<N_1}\frac{\partial^{\ell}\delta^{\ell}[(\rho(\xi))^*]}
{\ell!}+T_{N_1}(\xi,y)^*
\end{align*}
where
$$T_{N_1}(\xi,y)=\frac{1}{(2\pi)^n}\int_{\mathbb R^n}\!\int_{\mathbb R^n}
e^{-ix\cdot y}\alpha_x(R_{N_1}(\xi,y))\,\mathrm dx\,\mathrm dy$$
It remains to show that this symbol is of order $M$.
Obviously,
$$\sum_{N\le\ell<N_1}\frac{\partial^{\ell}\delta^{\ell}[(\rho(\xi))^*]}
{\ell!}\in S^{M-N},$$
so we need to show that the remainder is of order $M-N_1$.
Note that $$\sigma(P_{\rho}^*)(\xi)
-\sum_{|\ell|<N_1}\frac{\partial^{\ell}\delta^{\ell}[(\rho(\xi))^*]}
{\ell!}=T_{N_1}(\xi,y)^*.$$
Integrating by parts, we get
\begin{align*}
&\int_{\mathbb R^n}\!\int_{\mathbb R^n}\!
e^{-ix\cdot y}\frac{(-y)^{\ell}}{\ell!}\alpha_x\left(
\int_0^1\!(1-\gamma)^{N_1-1}(\partial^{\ell}\rho)(\xi-y\gamma)
\,\mathrm d\gamma\right)\,\mathrm dx\,\mathrm dy \\
&= \frac{1}{\ell!}\int_{\mathbb R^n}\!\int_{\mathbb R^n}\!
e^{-ix\cdot y}(-D_x)^{\ell}\alpha_x\left(
\int_0^1\!(1-\gamma)^{N_1-1}(\partial^{\ell}\rho)(\xi-y\gamma)
\,\mathrm d\gamma\right)\,\mathrm dx\,\mathrm dy \\
&= \frac{1}{\ell!}\int_{\mathbb R^n}\!\int_{\mathbb R^n}\!
e^{-ix\cdot y}(-\delta)^{\ell}\alpha_x\left(
\int_0^1\!(1-\gamma)^{N_1-1}(\partial^{\ell}\rho)(\xi-y\gamma)
\,\mathrm d\gamma\right)\,\mathrm dx\,\mathrm dy \\
&= \frac{1}{\ell!}\int_{\mathbb R^n}\!\int_{\mathbb R^n}\!
e^{-ix\cdot y}
\int_0^1\!(1-\gamma)^{N_1-1}
(-\delta)^{\ell}\alpha_x((\partial^{\ell}\rho)(\xi-y\gamma))
\,\mathrm d\gamma\,\mathrm dx\,\mathrm dy \\
&= \frac{1}{\ell!}\int_{\mathbb R^n}\!\int_{\mathbb R^n}\!
e^{-ix\cdot y}(-1)^{\ell}
\int_0^1\!(1-\gamma)^{N_1-1}
\alpha_x((\delta^{\ell}\partial^{\ell}\rho)(\xi-y\gamma))
\,\mathrm d\gamma\,\mathrm dx\,\mathrm dy
\end{align*}
where, for arbitrary $a=\sum_ma_m\prod_jU_j^{m_j}$,
\begin{align*}
(-D_x)^{\ell}\alpha_x(a)
&= (-D_x)^{\ell}\alpha_x\left(\sum_ma_m\prod_{j=1}^nU_j^{m_j}\right) \\
&= (-D_x)^{\ell}\sum_ma_me^{ix\cdot m}\prod_{j=1}^nU_j^{m_j} \\
&= \sum_ma_m(-m)^{\ell}e^{ix\cdot m}\prod_{j=1}^nU_j^{m_j} \\
&= (-\delta)^{\ell}\alpha_x(a).
\end{align*}
Since $\rho\in S^M$ and $|\ell|=N_1$ we have
$\delta^{\ell}\partial^{\ell}\rho\in S^{M-N_1}$.
We get the boundedness of
$$\ell^{N_1-M}(\xi)\int_{\mathbb R^n}\!\int_{\mathbb R^n}\!
e^{-ix\cdot y}\alpha_x(R_{N_1}(\xi,y))\,\mathrm dx\,\mathrm dy$$
because $\partial^{\ell}R_{N_1}(y,\xi)$ is the rest of index $N_1$
in the Taylor expansion of $\partial^{\ell}(\xi-y\gamma)$ for which
one has $\partial^{\ell}\rho\in S^{M-N_1}$.
\end{proof}
\section{Asymptotic formula for the symbol of a product
of two pseudo-differential operators}
Next we prove the formula for the product or composition of symbols
for the noncommutative $n$ torus, adapting the proof of
Theorem 7.1 of \cite{wong}.
\begin{thm}\label{prod}
Suppose that $P$ is a pseudodifferential operator with symbol
$\sigma(P)=\rho=\rho(\xi)$ of order $M_1$, and $Q$ is a pseudodifferential
operator with symbol $\sigma(Q)=\phi=\phi(\xi)$ of order $M_2$.
Then the symbol of the product $QP$ is of order $M_1+M_2$ and satisfies
$$\sigma(QP)\sim\sum_{\ell\in\mathbb Z_{\ge 0}^n}
\frac{1}{\ell_1!\ell_2!}\partial^{\ell}\phi(\xi)\delta^{\ell}\rho(\xi),$$
where $\partial^{\ell}:=\prod_j\partial_j^{\ell_j}$
and $\delta^{\ell}:=\prod_j\delta_j^{\ell_j}$.
\end{thm}
\begin{proof}
We want to show that if $\rho:\mathbb R^n\rightarrow A_{\theta}^{\infty}$
is of order $M_1$ and $\phi:\mathbb R^n\rightarrow A_{\theta}^{\infty}$
is of order $M_2$, $P_{\phi}\circ P_{\rho}=P_{\mu}$
where $\mu$ is of order $M_1+M_2$ and has asymptotic expansion
$$\mu\sim\sum_{\ell}\frac{1}{\ell!}
\partial^{\ell}\phi(\xi)\delta^{\ell}\rho(\xi).$$
Let $\{\varphi_k\}$ be the partition of unity constructed in Theorem 6.1 of
\cite{wong} and define $\phi_k(\xi):=\phi(\xi)\varphi_k(\xi)$.
We have $$P_{\phi_k}(a)=\frac{1}{(2\pi)^n}
\int_{\mathbb R^n}\!\int_{\mathbb R^n}\!
e^{is\cdot\xi}\phi_k(\xi)\alpha_s(a)\,\mathrm ds\,\mathrm d\xi$$
Summing over $k$ from zero to infinity and applying Fubini's Theorem, we get
\begin{align*}
\sum_{k=0}^{\infty}P_{\phi_k}(a)
&=\frac{1}{(2\pi)^n}\int_{\mathbb R^n}\!\int_{\mathbb R^n}\!
e^{-is\cdot\xi}\sum_{k=0}^{\infty}\phi_k(\xi)\alpha_s(a)
\,\mathrm ds\,\mathrm d\xi \\
&=\frac{1}{(2\pi)^n}\int_{\mathbb R^n}\!\int_{\mathbb R^n}\!
e^{-is\cdot\xi}\phi(\xi)\alpha_s(a)
\,\mathrm ds\,\mathrm d\xi,
\end{align*}
so $$P_{\phi}(a)=\sum_{k=0}^{\infty}P_{\phi_k}(a)$$ and the convergence
of the series is absolute and uniform for all $a\in A_{\theta}^{\infty}$.
We want to compute the symbol of $P_{\phi}\circ P_{\rho}$, but issues with
convergence of integrals make it so we need to compute the symbol of
$P_{\phi_k}\circ P_{\rho}$. Let $a\in A_{\theta}^{\infty}$ be arbitrary.
Applying $P_{\phi_k}\circ P_{\rho}$ we get
\begin{align*}
P_{\phi_k}(P_{\rho}(a))
&=\frac{1}{(2\pi)^n}\int_{\mathbb R^n}\!\int_{\mathbb R^n}\!
e^{-is\cdot\xi}\phi_k(\xi)\alpha_s(P_{\rho}(a))\,\mathrm ds\,\mathrm d\xi \\
&=\frac{1}{(2\pi)^n}\int_{\mathbb R^n}\!\int_{\mathbb R^n}\!
e^{-is\cdot\xi}\phi_k(\xi)\alpha_s\!\!\left(
\frac{1}{(2\pi)^n}\int_{\mathbb R^n}\!\int_{\mathbb R^n}\!
e^{-it\cdot\eta}\rho(\eta)\alpha_t(a)\,\mathrm dt\,\mathrm d\eta
\right)\mathrm ds\,\mathrm d\xi \\
&=\frac{1}{(2\pi)^{2n}}\int_{\mathbb R^n}\!\int_{\mathbb R^n}\!\!
\left\{\int_{\mathbb R^n}\!\int_{\mathbb R^n}\!
e^{-is\cdot\xi-it\cdot\eta}\phi_k(\xi)\alpha_s(\rho(\eta))\alpha_{s+t}(a)
\,\mathrm dt\,\mathrm d\eta\right\}\mathrm ds\,\mathrm d\xi \\
&=\frac{1}{(2\pi)^{2n}}\int_{\mathbb R^n}\!\int_{\mathbb R^n}\!\!
\left\{\int_{\mathbb R^n}\!\int_{\mathbb R^n}\!
e^{-ix\cdot\xi-i(y-x)\cdot\eta}\phi_k(\xi)\alpha_x(\rho(\eta))\alpha_y(a)
\,\mathrm dy\,\mathrm d\eta\right\}\mathrm dx\,\mathrm d\xi \\
&=\frac{1}{(2\pi)^{2n}}\int_{\mathbb R^n}\!\int_{\mathbb R^n}\!\!
\left\{\int_{\mathbb R^n}\!\int_{\mathbb R^n}\!
e^{-ix\cdot(\xi-\eta)-iy\cdot\eta}\phi_k(\xi)\alpha_x(\rho(\eta))\alpha_y(a)
\,\mathrm dy\,\mathrm d\eta\right\}\mathrm dx\,\mathrm d\xi \\
&=\frac{1}{(2\pi)^{2n}}\int_{\mathbb R^n}\!\int_{\mathbb R^n}\!\!
\left\{\int_{\mathbb R^n}\!\int_{\mathbb R^n}\!e^{-ix\cdot\sigma-iy\cdot\tau}
\phi_k(\sigma+\tau)\alpha_x(\rho(\tau))\alpha_y(a)
\,\mathrm dy\,\mathrm d\tau\right\}\mathrm dx\,\mathrm d\sigma \\
&=\frac{1}{(2\pi)^n}\int_{\mathbb R^n}\!\int_{\mathbb R^n}\!
e^{-iy\cdot\tau}\left\{
\frac{1}{(2\pi)^n}\int_{\mathbb R^n}\!\int_{\mathbb R^n}\!
e^{-ix\cdot\sigma}\phi_k(\sigma+\tau)\alpha_x(\rho(\tau))
\,\mathrm dx\,\mathrm d\sigma\right\}\alpha_y(a)\,\mathrm dy\,\mathrm d\tau \\
&=\frac{1}{(2\pi)^n}\int_{\mathbb R^n}\!\int_{\mathbb R^n}\!
e^{-iy\cdot\tau}\mu_k(\tau)\alpha_y(a)\,\mathrm dy\,\mathrm d\tau
\end{align*}
where
$$\mu_k(\tau)=\frac{1}{(2\pi)^n}\int_{\mathbb R^n}\!\int_{\mathbb R^n}\!
e^{-ix\cdot\sigma}\phi_k(\sigma+\tau)\alpha_x(\rho(\tau)))
\,\mathrm dx\,\mathrm d\sigma,$$
having done the changes of variables $(x,y)=(s,s+t)$ and
$(\sigma,\tau)=(\xi-\eta,\eta)$ and applied Proposition \ref{osc}.
This suggests that
$$P_{\phi}(P_{\rho}(a))
=\frac{1}{(2\pi)^n}\int_{\mathbb R^n}\!\int_{\mathbb R^n}\!
e^{-iy\cdot\tau}\mu(\tau)\alpha_y(a)\,\mathrm dy\,\mathrm d\tau$$
where $\mu(\tau)=\sum_{k=0}^{\infty}\mu_k(\tau)$.
We need to show that $\mu$ is a symbol in $S^{M_1+M_2}$ and
has our desired asymptotic expression.
Define $\mu_k$ by
$$\mu_k(\xi)=\frac{1}{(2\pi)^2}\int_{\mathbb R^n}\!\int_{\mathbb R^n}\!
e^{-ix\cdot y}\phi_k(\xi+y)\alpha_x(\rho(\xi))\,\mathrm dx\,\mathrm dy$$
for all $\xi\in\mathbb R^n$.
By Taylor's formula with integral remainder given in Theorem 6.3 of
\cite{wong}, we get
$$\phi_k(\xi+y)
=\sum_{|\ell|<N_1}\frac{y^{\ell}}{\ell!}
(\partial^{\ell}\phi_k)(\xi)+R_{N_1}(y,\xi)$$
where
\begin{equation}
R_{N_1}(y,\xi)=N_1\sum_{|\ell|=N_1}
\frac{y^{\ell}}{\ell!}
\int_0^1\!(1-\gamma)^{N_1-1}(\partial^{\ell}\phi_k)
(\xi+\gamma y)\,\mathrm d\gamma
\end{equation}
for all $y,\xi\in\mathbb R^2$.
Substituting back into our expression for $\mu_k(\xi)$ we get
$$\mu_k(\xi)=\frac{1}{(2\pi)^n}\int_{\mathbb R^n}\!\int_{\mathbb R^n}\!
e^{-ix\cdot y}\sum_{|\ell|<N_1}\frac{y^{\ell}}{\ell!}
(\partial^{\ell}\phi_k)(\xi)\alpha_x(\rho(\xi))
\,\mathrm dx\,\mathrm dy+T_{N_1}^{(k)}(\xi)$$
where
$$T_{N_1}^{(k)}(\xi)=\frac{1}{(2\pi)^n}\int_{\mathbb R^n}\!\int_{\mathbb R^n}\!
e^{-ix\cdot y}R_{N_1}(y,\xi)\alpha_x(\rho(\xi))\,\mathrm dx\,\mathrm dy.$$
Expressing $\rho(\xi)$ as $\rho(\xi)=\sum_m\rho_m(\xi)\prod_{j=1}^nU_j^{m_j}$, we see
that $$\alpha_x(\rho(\xi))=\sum_m\rho_m(\xi)e^{ix\cdot m}\prod_{j=1}^nU_j^{m_j}$$ so
\begin{align*}
\mu_k(\xi)-T_{N_1}^{(k)}(\xi)
&=\sum_{|\ell|<N_1}\sum_m\frac{1}{(2\pi)^n}
\int_{\mathbb R^n}\!\int_{\mathbb R^n}\!
e^{-ix\cdot y}\frac{y^{\ell}}{\ell!}
(\partial^{\ell}\phi_k)(\xi)
\rho_m(\xi)e^{ix\cdot m}\prod_{j=1}^nU_j^{m_j}\,\mathrm dx\,\mathrm dy \\
&=\sum_{|\ell|<N_1}\frac{1}{\ell!}
(\partial^{\ell}\phi_k)(\xi)
\sum_m\rho_m(\xi)\prod_{j=1}^nU_j^{m_j}\frac{1}{(2\pi)^n}
\int_{\mathbb R^n}\!\int_{\mathbb R^n}\!e^{-ix\cdot(y-m)}y^{\ell}
\,\mathrm dx\,\mathrm dy \\
&=\sum_{|\ell|<N_1}\frac{1}{\ell!}(\partial^{\ell}\phi_k)(\xi)
\sum_m\rho_m(\xi)\prod_{j=1}^nU_j^{m_j}m^{\ell} \\
&=\sum_{|\ell|<N_1}\frac{1}{\ell!}
(\partial^{\ell}\phi_k)(\xi)(\delta^{\ell}\rho)(\xi).
\end{align*}
Let $\mu(\xi)=\sum_{k=0}^{\infty}\mu_k(\xi)$. It remains to show that
$\mu$ is of order $M_1+M_2$. Obviously,
$$\sum_{N\le|\ell|<N_1}\frac{y^{\ell}}{\ell!}
(\partial^{\ell}\phi)(\xi)(\delta^{\ell}\rho)(\xi)\in S^{M_1+M_2-N},$$
so we just need to show that the remainder is of order $M_1+M_2-N_1$.
Note that
$$\mu(\xi)-\sum_{|\ell|<N_1}\frac{y^{\ell}}{\ell!}
(\partial^{\ell}\phi)(\xi)(\delta^{\ell}\rho)(\xi)
=\sum_{k=0}^{\infty}T_{N_1}^{(k)}(\xi)$$
where
$$T_{N_1}^{(k)}(\xi)=\frac{1}{(2\pi)^n}\int_{\mathbb R^n}\!\int_{\mathbb R^n}\!
e^{-ix\cdot y}R_{N_1}(y,\xi)\alpha_x(\rho(\xi))\,\mathrm dx\,\mathrm dy$$
and
$$R_{N_1}(y,\xi)
=N_1\sum_{|\ell|=N_1}\frac{y^{\ell}}{\ell!}
\int_0^1\!(1-\gamma)^{N_1-1}(\partial^{\ell}\phi_k)(\xi+\gamma y)
\,\mathrm d\gamma.$$
Integrating by parts, we get
\begin{align*}
&\int_{\mathbb R^n}\!\int_{\mathbb R^n}\!
e^{-ix\cdot y}\frac{y^{\ell}}{\ell!}\int_0^1(1-\gamma)^{N_1-1}
(\partial^{\ell}\phi_k)(\xi+\gamma y)\,\mathrm d\gamma\alpha_x(\rho(\xi))
\,\mathrm dx\,\mathrm dy \\
&= \frac{1}{\ell!}\int_{\mathbb R^n}\!\int_{\mathbb R^n}\!
e^{-ix\cdot y}y^{\ell}\int_0^1(1-\gamma)^{N_1-1}(\partial^{\ell}(\phi_k)
(\xi+\gamma y)\,\mathrm d\gamma\alpha_x(\rho(\xi))\,\mathrm dx\,\mathrm dy \\
&= \frac{1}{\ell!}\int_{\mathbb R^n}\!\int_{\mathbb R^n}\!
e^{-ix\cdot y}\int_0^1(1-\gamma)^{N_1-1}(\partial^{\ell}(\phi_k)
(\xi+\gamma y)\,\mathrm d\gamma D_x^{\ell}\alpha_x(\rho(\xi))
\,\mathrm dx\,\mathrm dy \\
&= \frac{1}{\ell!}\int_{\mathbb R^n}\!\int_{\mathbb R^n}\!
e^{-ix\cdot y}\int_0^1(1-\gamma)^{N_1-1}(\partial^{\ell}(\phi_k)
(\xi+\gamma y)\,\mathrm d\gamma\delta^{\ell}\alpha_x(\rho(\xi))
\,\mathrm dx\,\mathrm dy \\
&= \frac{1}{\ell!}\int_{\mathbb R^n}\!\int_{\mathbb R^n}\!
e^{-ix\cdot y}\int_0^1(1-\gamma)^{N_1-1}(\partial^{\ell}(\phi_k)
(\xi+\gamma y)\,\mathrm d\gamma\alpha_x((\delta^{\ell}\rho)(\xi))
\,\mathrm dx\,\mathrm dy
\end{align*}
where for arbitrary $a=\sum_ma_m\prod_jU_j^{m_j}$ we have
\begin{align*}
(D_x)^{\ell}\alpha_x(a)
&= (D_x)^{\ell}\alpha_x\left(\sum_ma_m\prod_{j=1}^nU_j^{m_j}\right) \\
&= (D_x)^{\ell}\sum_ma_me^{ix\cdot m}\prod_{j=1}^nU_j^{m_j} \\
&= \sum_m a_m m^{\ell}e^{ix\cdot m}\prod_{j=1}^nU_j^{m_j} \\
&= \delta^{\ell}\alpha_x(a).
\end{align*}
Since $|\ell|=N_1$, we have $\partial^{\ell}(\phi_k)\in S^{M_2-N_1}$
and $\delta^{\ell}\rho\in S^{M_1}$. We get the boundedness of
$$\mu^{N_1-M_1-M_2}(\xi)\int_{\mathbb R^n}\!\int_{\mathbb R^n}\!
e^{-ix\cdot y}R_{N_1}(y,\xi)\alpha_x(\rho(\xi))\,\mathrm dx\,\mathrm dy$$
since $\delta^{\ell}\rho\in S^{M_1}$ and
$\partial^{\ell}R_{N_1}(y,\xi)$ is the rest of index $N_1$ in
Taylor's expansion of $\partial^{\ell}\phi_k(\xi+\gamma y)$ for which
on has $\partial^{\ell}\phi_k\in S^{M_2-N_1}$.
\end{proof}
\section{Sobolev spaces on the noncommutative $n$ torus}
Let $\lambda(\xi)=(1+\xi_1^2+\cdots+\xi_n^2)^{1/2}$.
Consider the following inner product on $A_{\theta}^{\infty}$.
\begin{define}
Define the Sobolev inner product
$\langle\cdot,\cdot\rangle_s:A_{\theta}^{\infty}\times
A_{\theta}^{\infty}\rightarrow\mathbb C$ by
$$\langle a,b\rangle_s:=
\langle P_{\lambda^s}(a),P_{\lambda^s}(b)\rangle
=\sum_m(1+|m_1|^2+\cdots+|m_n|^2)^s\overline{b_m}a_m.$$
\end{define}
Note that for $s=0$ this agrees with
$\langle\cdot,\cdot\rangle$.
This inner product induces the following norm.
\begin{define}
Define the Sobolev norm $||\cdot||_s:A_{\theta}^{\infty}\rightarrow\mathbb R_{\ge 0}$ by
$$||a||_s^2:=\langle P_{\lambda^s}(a),P_{\lambda^s}(a)\rangle
=\sum_m(1+|m_1|^2+\cdots+|m_n|^2)^s|a_m|^2.$$
\end{define}
Using this norm, we can define the analog of Sobolev spaces on the
noncommutative $n$ torus.
\begin{define}
Define the Sobolev space $H^s$ to be the completion of $A_{\theta}^{\infty}$
with respect to $||\cdot||_s$.
\end{define}
We can prove that a pseudo-differential operator of order $d\in\mathbb R$
continuously maps $H^s$ into $H^{s-d}$. However we must first prove the case
where $s=d$.
\begin{thm}\label{ws3}
Suppose $\rho\in S^d$. Then $||P_{\rho}(a)||_0\le C||a||_d$ for
some constant $C>0$ and $P_{\rho}$ defines a bounded operator
$P_{\rho}:H^d\rightarrow H^0$.
\end{thm}
\begin{proof}
Note that $\{\prod_jU_j^{m_j}:m\in\mathbb Z^n\}$
is an orthogonal basis with respect to
$\langle\cdot,\cdot\rangle_s$ which is orthonormal in the case $s=0$.
We have
$||\prod_jU_j^{m_j}||_0^2=1$,
$||\rho_m(\xi)\prod_jU_j^{m_j}||_0^2=|\rho_m(\xi)|^2$,
and $||\rho(\xi)||_0^2=\sum_m|\rho_m(\xi)|^2$.
Since $\rho$ is of order $d$, we have
$||\rho(\xi)||_0\le C_{\rho}(1+|\xi|)^d$, and
since $(1-|\xi|)^2\ge 0$ gives us $(1+|\xi|)^2\le 2(1+|\xi|^2)$, we have
$$||\rho(\xi)||_0^2\le C_{\rho}^2(1+|\xi|)^{2d}\le C_{\rho}^22^d(1+|\xi|^2)^d.$$
Let $k_{\rho}:=C_{\rho}^22^d$. Then we have
$$\sum_m|\rho_m(\xi)|^2\le k_{\rho}(1+|\xi|^2)^d.$$
Let $e_{s,m}:=(1+|m_1|^2+\cdots+|m_n|^2)^{-s/2}\prod_jU_j^{m_j}$ and
$E_s:=\{e_{s,m}\mid m\in\mathbb Z^n\}$. By definition we have $E_s$
orthonormal with respect to $\langle\cdot,\cdot\rangle_s$. It suffices to
prove this theorem for the case $a=e_{d,m}$ by the orthonormality of $E_d$
since
$$||P_{\rho}(a)||_0^2=\sum_m|a_m|^2||P_{\rho}(e_{d,m})||_0^2$$
and
$$||a||_d^2=\sum_m|a_m|^2||e_{d,m}||_d^2.$$
Since $||e_{d,m}||_d^2=1$, it suffices to show that
$$||P_{\rho}(e_{d,m})||_0^2\le K$$ for some constant $K>0$.
We have
\begin{align*}
||P_{\rho}(e_{d,m})||_0^2
&= ||\rho(m)e_{d,m}||_0^2 \\
&= ||\rho(m)(1+|m_1|^2+\cdots+|m_n|^2)^{-d/2}\prod_{j=1}^nU_j^{m_j}||_0^2 \\
&= \left|\left|\sum_k
\rho_k(m)\prod_{j=1}^nU_j^{k_j}
(1+|m_1|^2+\cdots+|m_n|^2)^{-d/2}\prod_{j=1}^nU_j^{m_j}
\right|\right|_0^2 \\
&= (1+|m_1|^2+\cdots+|m_n|^2)^{-d}\left|\left|\sum_k
\rho_k(m)\prod_{j=1}^nU_j^{k_j}\prod_{j=1}^nU_j^{m_j}\right|\right|_0^2 \\
&= (1+|m_1|^2+\cdots+|m_n|^2)^{-d}\left|\left|\sum_k
\rho_k(m)w(m,k)\prod_j^nU_j^{m_j+k_j}
\right|\right|_0^2 \\
&= (1+|m_1|^2+\cdots+|m_n|^2)^{-d}\left|\left|\sum_{k,\ell}
\rho_{k-m}(m)w(m,k-m)\prod_{j=1}^nU_j^{k_j}
\right|\right|_0^2 \\
&= (1+|m_1|^2+\cdots+|m_n|^2)^{-d}
\sum_k|\rho_{k-m}(m)|^2 \\
&= (1+|m_1|^2+\cdots+|m_n|^2)^{-d}
\sum_{k,\ell}|\rho_k(m)|^2 \\
&\le (1+|m_1|^2+\cdots+|m_n|^2)^{-d}k_{\rho}(1+|m_1|^2+\cdots+|m_n|^2)^d
=k_{\rho}
\end{align*}
where $$w(k,m):=\prod_{j=1}^nU_j^{m_j}\prod_{j=1}^nU_j^{k_j}
\left(\prod_{j=1}^nU_j^{m_j+k_j}\right)\in S^1\subset\mathbb C$$
so our desired constant is $K=k_{\rho}=C_{\rho}^22^d$ and we are done.
\end{proof}
For the general case $s\ne d$ we need to prove a lemma saying that
$||\cdot||_s=||\cdot||_{s-t}\circ P_{\lambda^t}$.
\begin{lem}\label{sn3}
For any $a\in A_{\theta}^{\infty}$ and $s,t\in\mathbb R$,
$a\in H^s$ if and only if
$P_{\lambda^t}(a)\in H^{s-t}$ with $||a||_s=||P_{\lambda^t}(a)||_{s-t}$.
\end{lem}
\begin{proof}
Suppose that $a\in H^s$ or $P_{\lambda^t}(a)\in H^{s-t}$. Then
\begin{align*}
||P_{\lambda^t}(a)||_{s-t}^2
&= \sum_{m}(1+|m_1|^2+\cdots+|m_n|^2)^{s-t}\lambda^{2t}(m)|a_m|^2 \\
&= \sum_{m}(1+|m_1|^2+\cdots+|m_n|^2)^{s-t}(1+|m_1|^2+\cdots+|m_n|^2)^t|a_m|^2\\
&= \sum_{m}(1+|m_1|^2+\cdots+|m_n|^2)^s|a_m|^2 \\
&= ||a||_s
\end{align*}
so we know that $a\in H^s$ and $P_{\lambda^t}(a)\in H^{s-t}$.
\end{proof}
Then the general case follows quite easily.
\begin{cor}
Suppose $\rho\in S^d$. Then $||P_{\rho}(a)||_{s-d}\le C||a||_s$ for
some constant $C>0$ and $P_{\rho}$ defines a bounded operator
$P_{\rho}:H^s\rightarrow H^{s-d}$.
\end{cor}
\begin{proof}
By Lemma \ref{sn3}, we have
$||P_{\rho}(a)||_{s-d}=||P_{\lambda^{s-d}}(P_{\rho}(a))||_0$.
By Theorem \ref{prod}, the symbol
$\sigma(P_{\lambda^{s-d}}\circ P_{\rho})$ is of order $d+(s-d)=s$,
so Theorem \ref{ws3} gives us
$||P_{\lambda^{s-d}}(P_{\rho}(a))||_0\le C||a||_s$ for some constant $C>0$.
\end{proof}
We can also define an analog of the $C^k$ norm for the noncommutative $n$ torus.
\begin{define}
Define the $C^k$ norm
$||\cdot||_{\infty,k}:A_{\theta}^{\infty}\rightarrow\mathbb R_{\ge 0}$
as follows:
$$||a||_{\infty,k}:=\sum_{|\ell|\le k}||\delta^{\ell}(a)||_{C^*}$$
where the $C^*$ norm $||\cdot||_{C^*}$ is given by
$$||a||_{C^*}^2:=\sup\{|\lambda|:a^*a-\lambda\cdot 1\text{ not invertible}\}.$$
\end{define}
Since, for arbitrary $a=\sum_ma_m\prod_jU_j^{m_j}$,
\begin{align*}
D_s^{\ell}\alpha_s(a)
&=(-i\partial_s)^{\ell}\sum_me^{is\cdot m}a_m\prod_{j=1}^nU_j^{m_j} \\
&=\sum_m m^{\ell}e^{is\cdot m}a_m\prod_{j=1}^nU_j^{m_j} \\
&=\sum_m\delta^{\ell}e^{is\cdot m}a_m\prod_{j=1}^nU_j^{m_j} \\
&=\delta^{\ell}\alpha_s(a)
\end{align*}
we have $A_{\theta}^k=C^k$.

We can easily prove an analog of the Sobolev lemma as follows.
\begin{thm}
For $s>k+1$, $H^s\subseteq A_{\theta}^k$.
\end{thm}
\begin{proof}
First consider the case $k=0$.
Note that $||\cdot||_{\infty,0}=||\cdot||_{C^*}$
so for arbitrary $a_m\prod_jU_j^{m_j}$ we have
$$||a_m\prod_{j=1}^nU_j^{m_j}||_{\infty,0}^2=\sup\{|\lambda|:
|a_m|^2-\lambda\cdot 1\text{ not invertible}\}=|a_m|^2$$
and for arbitrary $a=\sum_ma_m\prod_jU_j^{m_j}$ we have
$$||a||_{\infty,0}^2\le\sum_m||a_m\prod_{j=1}^nU_j^{m_j}||_{\infty,0}^2
=\sum_m|a_m|^2=||a||_0^2$$ by the triangle inequality.
We have $$a=\sum_ma_m\lambda^s(m)\lambda^{-s}(m)\prod_{j=1}^nU_j^{m_j}$$
so by the Cauchy-Schwarz inequality we get
$$||a||_0^2\le||a||_s^2\sum_m(1+|m_1|^2+\cdots+|m_n|^2)^{-s}.$$
Since $2s>2$, $(1+|m_1|^2+\cdots+|m_n|^2)^{-s}$ is summable over
$m\in\mathbb Z^n$ and $||a||_0^2\le C||a||_s$. Thus we get
$||a||_{\infty,0}\le||a||_0\le C||a||_s$ and $H^s\subseteq A_{\theta}^0$.

Now suppose $k>0$. Using what we've proven for the previous case, we have
\begin{align*}
||\delta^{\ell}(a)||_{\infty,0} &\le C||\delta^{\ell}(a)||_{s-|\ell|} \\
&= C||\sum_m m^{\ell}a_m\prod_{j=1}^nU_j^{m_j}||_{s-|\ell|} \\
&< C||\sum_m(1+|m_1|^2+\cdots+|m_n|^2)^{|\ell|}
a_m\prod_{j=1}^nU_j^{m_j}||_{s-|\ell|} \\
&= C||P_{\lambda^{|\ell|}}(a)||_{s-|\ell|} \\
&= C||a||_s
\end{align*}
for $|\ell|\le k$ since $s-|\ell|\ge s-k>1$.
Therefore,
$$||a||_{\infty,k}=\sum_{|\ell|\le k}||\delta^{\ell}(a)||_{\infty,0}
\le\sum_{|\ell|\le k}C||a||_s
\le C||a||_s(k+1)(k+2)/2$$
and we get $H^s\subseteq A_{\theta}^k$.
\end{proof}
We get the following corollary.
\begin{cor}
$\displaystyle\bigcap_{s\in\mathbb R}H^s=A_{\theta}^{\infty}$.
\end{cor}
\begin{proof}
Suppose $a\in\bigcap_{s\in\mathbb R}H^s$. Then for any
$k\in\mathbb Z_{\ge 0}$, $a\in H^{k+2}$, so by the theorem we just proved,
$a\in A_{\theta}^k$. Consequently $a\in A_{\theta}^{\infty}$, so
$\bigcap_{s\in\mathbb R}H^s\subseteq A_{\theta}^{\infty}$.

Suppose $a\in A_{\theta}^{\infty}$. Then since
$H^s$ is the completion of $A_{\theta}^{\infty}$ with respect to $||\cdot||_s$,
$A_{\theta}^{\infty}\subseteq H^s$ for all $s\in\mathbb R$, and
$A_{\theta}^{\infty}\subseteq\bigcap_{s\in\mathbb R}H^s$.
\end{proof}
We can also prove an analog of the Rellich lemma for the noncommutative
$n$ torus.
\begin{thm}
Let $\{a_N\}\in A_{\theta}^{\infty}$ be a sequence. Suppose that there
is a constant $C$ so that $||a_N||_s\le C$ for all $N$. Let $s>t$. Then
there is a subsequence $\{a_{N_j}\}$ that converges in $H^t$.
\end{thm}
\begin{proof}
Let $e_{s,m}:=(1+|m_1|^2+\cdots+|m_n|^2)^{-s/2}\prod_{j=1}^nU_j^{m_j}$ and
$E_s:=\{e_{s,m}\mid m\in\mathbb Z^n\}$. $E_s$ is an orthonormal basis
with respect to $\langle\cdot,\cdot\rangle_s$, so we can write
$a_N:=\sum_ka_{N,k}e_{s,k}$. Then
$$|a_{N,k}|^2\le\sum_k|a_{N,k}|^2\le C^2$$
and $|a_{N,k}|\le C$.
Applying the Arzela-Ascoli theorem to $\{a_{N,k}\}$ for some fixed
$k$, we can get a subsequence $\{a_{N_j,k}\}$ of $\{a_{N,k}\}$
such that for any $\epsilon>0$ there exists $M(\epsilon)\in\mathbb N$
such that $|a_{N_i,k}-a_{N_j,k}|<\epsilon$ whenever
$i,j\ge M(\epsilon)$.
Do this for all $|k_1|^2+\cdots+|k_n|^2\le r$,
replacing $\{a_N\}$ with $\{a_{N_j}\}$
each time. Then we get a subsequence $\{a_{N_j}\}$ of $\{a_N\}$
such that for any $\epsilon>0$ there exists $M(\epsilon)\in\mathbb N$
such that, for all $|k_1|^2+\cdots+|k_n|^2\le r$,
$|a_{N_i,k}-a_{N_j,k}|<\epsilon$ whenever $i,j\ge M(\epsilon)$.
Now consider the sum $$||a_{N_i}-a_{N_j}||_t^2
=\sum_k|a_{N_i,k}-a_{N_j,k}|^2(1+|k_1|^2+\cdots+|k_n|^2)^{t-s}.$$
Decompose it into two parts: one where $|k_1|^2+\cdots+|k_n|^2>r^2$ and one
where $|k_1|^2+\cdots+|k_n|^2\le r$. On $|k_1|^2+\cdots+|k_n|^2>r^2$ we estimate
$$(1+|k_1|^2+\cdots+|k_n|^2)^{t-s}<(1+r^2)^{t-s}$$ so that
\begin{align*}
\sum_{|k_1|^2+\cdots+|k_n|^2\ge r^2}|a_{N_i,k}-a_{N_j,k}|^2
(1+|k_1|^2+\cdots+|k_n|^2)^{t-s}
&<(1+r^2)^{t-s}\sum_k|a_{N_i,k}-a_{N_j,k}|^2 \\
&\le2C^2(1+r^2)^{t-s}.
\end{align*}
If $\epsilon>0$ is given, we choose $r$ so that $2C^2(1+r^2)^{t-s}<\epsilon$.
The remaining part of the sum is over $|k_1|^2+\cdots+|k_n|^2\le r^2$ and
can be bounded above by
$\epsilon':=\epsilon-2C^2(1+r^2)^{t-s}$ if
$i,j\ge M(\sqrt[n]{\epsilon'}/(2r+1))$
because a ball of radius $r$ centered at the origin is contained in a cube
of side length $2r$ that has $(2r+1)^n$ lattice points.
Then the total sum is bounded above by $\epsilon$, and we are done.
\end{proof}
\section{The pseudodifferential calculus on finitely generated projective
modules over the noncommutative $n$ torus}
We can generalize these results to arbitrary finitely generated projective
right modules over the noncommutative $n$ torus following p.~553 of \cite{ncg},
which considers finitely generated projective modules over an arbitrary unital
$*$-algebra. Let $E$ be a finitely generated projective right
$A_{\theta}^{\infty}$-module. Since $E$ is a finitely generated
projective right $A_{\theta}^{\infty}$-module, we can write $E$ as a direct
summand $E=(A_{\theta}^{\infty})^re$ of a free module $(A_{\theta}^{\infty})^r$
with direct complement
$F=(A_{\theta}^{\infty})^r(\mathrm{id}-e)$,
where the idempotent $e\in M_r(A_{\theta}^{\infty})$ is self-adjoint.
Consider an $r\times r$ matrix valued symbol $\rho=(\rho_{j,k})$ where
$\rho_{j,k}:\mathbb R^n\rightarrow A_{\theta}^{\infty}$ are scalar symbols
and $\rho_{j,k}\in S^d$. Define the operator $P_{\rho}:E\rightarrow E$ as
follows:
$$P_{\rho}(\vec{a}):=(2\pi)^{-n}\int_{\mathbb R^n}\!\int_{\mathbb R^n}\!
e^{-is\cdot\xi}\rho(\xi)\alpha_s(\vec{a})\,\mathrm ds\,\mathrm d\xi.$$
Define the inner product
$\langle\vec{a},\vec{b}\rangle:E\times E\rightarrow\mathbb C$
sending $(\vec{a},\vec{b})\mapsto\tau(\vec{b}^*\vec{a})$.
Since $$P_{\rho}(\vec{a})_j
=(2\pi)^{-n}\int_{\mathbb R^n}\!\int_{\mathbb R^n}\!e^{-is\cdot\xi}
\sum_{k=1}^r\rho_{j,k}(\xi)\alpha_s(a_k)\,\mathrm ds\,\mathrm d\xi,$$
Lemma \ref{opn} generalizes to $E$ as follows after applying it to each
component:
$$P_{\rho}(\vec{a})=\sum_m\rho(m)\vec{a}_m\prod_{j=1}^nU_j^{m_j}.$$
Theorems \ref{adj} and \ref{prod} generalize as follows.
\begin{thm}\label{mvs2}
\begin{enumerate}[(a)]
\item For a pseudodifferential operator $P$ with $r\times r$ matrix valued
symbol $\sigma(P)=\rho=\rho(\xi)$, the
symbol of the adjoint $P^*$ satisfies
$$\sigma(P^*)\sim\sum_{(\ell_1,\ldots,\ell_n)\in(\mathbb Z_{\ge 0})^n}
\frac{\partial_1^{\ell_1}\cdots\partial_n^{\ell_n}
\delta_1^{\ell_1}\cdots\delta_n^{\ell_n}(\rho(\xi))^*}{\ell_1!\cdots\ell_n!}.$$
\item If $Q$ is a pseudodifferential operator with $r\times r$ matrix valued symbol
$\sigma(Q)=\rho'=\rho'(\xi)$,
then the product $PQ$ is also a pseudodifferential operator and has symbol
$$\sigma(PQ)\sim\sum_{(\ell_1,\ldots,\ell_n)\in(\mathbb Z_{\ge 0})^n}
\frac{\partial_1^{\ell_1}\cdots\partial_n^{\ell_n}(\rho(\xi))
\delta_1^{\ell_1}\cdots\delta_n^{\ell_n}(\rho'(\xi))}{\ell_1!\cdots\ell_n!}.$$
\end{enumerate}
\end{thm}
\begin{proof}
First let's prove part (a). Let $\rho$ be an $r\times r$ matrix valued symbol
of order $M$ and $\vec{a},\vec{b}\in E$. We have
\begin{align*}
\langle P_{\rho}(\vec{a}),\vec{b}\rangle
&= \tau\left(\vec{b}^*\frac{1}{(2\pi)^n}\int_{\mathbb R^n}\!\int_{\mathbb R^n}\!
e^{-is\cdot\xi}\rho(\xi)\alpha_s(\vec{a})\,\mathrm ds\,\mathrm d\xi\right) \\
&= \tau\left(\left(\frac{1}{(2\pi)^n}\int_{\mathbb R^n}\!\int_{\mathbb R^n}\!
e^{+is\cdot\xi}\alpha_{-s}(\rho(\xi)^*\vec{b})\,\mathrm ds\,\mathrm d\xi
\right)^*\vec{a}\right) \\
&= \langle\vec{a},P_{\rho}^*(\vec{b})\rangle
\end{align*}
where
\begin{align*}
P_{\rho}^*(\vec{b})_j
&= \frac{1}{(2\pi)^n}\int_{\mathbb R^n}\!\int_{\mathbb R^n}\!e^{+is\cdot\xi}
\sum_{k=1}^r\alpha_{-s}(\rho(\xi)_{j,k}^*b_k)\,\mathrm ds\,\mathrm d\xi \\
&= \frac{1}{(2\pi)^n}\int_{\mathbb R^n}\!\int_{\mathbb R^n}\!e^{+is\cdot\xi}
\sum_{k=1}^r\alpha_{-s}(\rho(\xi)_{j,k}^*)\alpha_{-s}(b_k)
\,\mathrm ds\,\mathrm d\xi \\
&= \sum_{m,p}(\rho_m(p-m)\prod_{h=1}^nU_h^{m_h})^*(b_p\prod_{h=1}^nU_h^{p_h})
\end{align*}
so
\begin{align*}
\sigma(P_{\rho}^*)(\xi)
&= \left[\sum_m\rho_m(\xi-m)\prod_{h=1}^nU_h^{m_h}\right]^* \\
&= \left[\frac{1}{(2\pi)^n}\int_{\mathbb R^n}\!\int_{\mathbb R^n}\!
e^{-ix\cdot y}\sum_m\rho_m(\xi-y)\alpha_x\left(\prod_{h=1}^nU_h^{m_h}\right)
\,\mathrm dx\,\mathrm dy\right]^* \\
&= \left[\frac{1}{(2\pi)^n}\int_{\mathbb R^n}\!\int_{\mathbb R^n}\!
e^{-ix\cdot y}\alpha_x(\rho(\xi-y))\,\mathrm dx\,\mathrm dy\right]^*.
\end{align*}
The rest of the proof reduces to the $r=1$ case, applying it to each entry
in $\rho=(\rho_{j,k})$.

We proceed to part (b). Let $\rho$ be an $r\times r$ matrix valued symbol
of order $m_1$ and $\phi$ be an $r\times r$ matrix valued symbol of order
$m_2$. Let $\{\varphi_k\}$ be a partition of unity and define
$\phi_k(\xi):=\phi(\xi)\varphi_k(\xi)$. Let
$\vec{a}\in E$. We have
\begin{align*}
P_{\phi_k}(P_{\rho}(\vec{a})) &=
\frac{1}{(2\pi)^n}\int_{\mathbb R^n}\!\int_{\mathbb R^n}\!
e^{-is\cdot\xi}\phi_k(\xi)\alpha_s(P_{\rho}(\vec{a}))
\,\mathrm ds\,\mathrm d\xi \\
&= \frac{1}{(2\pi)^n}\int_{\mathbb R^n}\!\int_{\mathbb R^n}\!
e^{-is\cdot\xi}\phi_k(\xi)\alpha_s\left(
\frac{1}{(2\pi)^n}\int_{\mathbb R^n}\!\int_{\mathbb R^n}\!e^{-it\cdot\eta}
\rho(\eta)\alpha_t(\vec{a})\,\mathrm dt\,\mathrm d\eta\right)
\,\mathrm ds\,\mathrm d\xi \\
&= \frac{1}{(2\pi)^{2n}}\int_{\mathbb R^n}\!\int_{\mathbb R^n}\!
\left\{\int_{\mathbb R^n}\!\int_{\mathbb R^n}\!
e^{-is\cdot\xi-it\cdot\eta}\phi_k(\xi)\alpha_s(\rho(\eta))\alpha_{s+t}(\vec{a})
\,\mathrm dt\,\mathrm d\eta\right\}\,\mathrm ds\,\mathrm d\xi \\
&= \frac{1}{(2\pi)^{2n}}\int_{\mathbb R^n}\!\int_{\mathbb R^n}\!
\left\{\int_{\mathbb R^n}\!\int_{\mathbb R^n}\!
e^{-ix\cdot\xi-i(y-x)\cdot\eta}\phi_k(\xi)\alpha_x(\rho(\eta))\alpha_y(\vec{a})
\,\mathrm dy\,\mathrm d\eta\right\}\,\mathrm dx\,\mathrm d\xi \\
&= \frac{1}{(2\pi)^{2n}}\int_{\mathbb R^n}\!\int_{\mathbb R^n}\!
\left\{\int_{\mathbb R^n}\!\int_{\mathbb R^n}\!
e^{-ix\cdot(\xi-\eta)-iy\cdot\eta}
\phi_k(\xi)\alpha_x(\rho(\eta))\alpha_y(\vec{a})
\,\mathrm dy\,\mathrm d\eta\right\}\,\mathrm dx\,\mathrm d\xi \\
&= \frac{1}{(2\pi)^{2n}}\int_{\mathbb R^n}\!\int_{\mathbb R^n}\!
\left\{\int_{\mathbb R^n}\!\int_{\mathbb R^n}\!
e^{-ix\cdot\sigma-iy\cdot\eta}
\phi_k(\sigma+\tau)\alpha_x(\rho(\tau))\alpha_y(\vec{a})
\,\mathrm dy\,\mathrm d\tau\right\}\,\mathrm dx\,\mathrm d\sigma \\
&= \frac{1}{(2\pi)^{2n}}\int_{\mathbb R^n}\!\int_{\mathbb R^n}\!e^{-iy\cdot\tau}
\left\{\int_{\mathbb R^n}\!\int_{\mathbb R^n}\!e^{-ix\cdot\sigma}
\phi_k(\sigma+\tau)\alpha_x(\rho(\tau))\,\mathrm dx\,\mathrm d\sigma
\right\}\alpha_y(\vec{a})\,\mathrm dy\,\mathrm d\tau \\
&= \frac{1}{(2\pi)^n}\int_{\mathbb R^n}\!\int_{\mathbb R^n}\!e^{-iy\cdot\tau}
\lambda_k(\tau)\alpha_y(\vec{a})\,\mathrm dy\,\mathrm d\tau
\end{align*}
where
$$\lambda_k(\tau)=\frac{1}{(2\pi)^n}\int_{\mathbb R^n}\!\int_{\mathbb R^n}\!
e^{-ix\cdot\sigma}\phi_k(\sigma+\tau)\alpha_x(\rho(\tau))
\,\mathrm dx\,\mathrm d\sigma$$
so
$$P_{\phi}(P_{\rho}(\vec{a}))
=\frac{1}{(2\pi)^n}\int_{\mathbb R^n}\!\int_{\mathbb R^n}\!e^{-iy\cdot\tau}
\lambda(\tau)\alpha_y(\vec{a})\,\mathrm dy\,\mathrm d\tau$$
where $\lambda(\tau)=\sum_{k=0}^{\infty}\lambda_k(\tau)$.

Let $$\lambda_k(\xi):=\frac{1}{(2\pi)^n}\int_{\mathbb R^n}\!\int_{\mathbb R^n}\!
e^{-ix\cdot y}\phi_k(\xi+y)\alpha_x(\rho(\xi))\,\mathrm dx\,\mathrm dy.$$
Since
\begin{align*}
\lambda_k(\xi)_{\alpha,\gamma}
&= \frac{1}{(2\pi)^n}\int_{\mathbb R^n}\!\int_{\mathbb R^n}\!e^{-ix\cdot y}
\sum_{\beta=1}^r\phi_k(\xi+y)_{\alpha,\beta}\alpha_x(\rho(\xi)_{\beta,\gamma})
\,\mathrm dx\,\mathrm dy \\
&= \sum_{\beta=1}^r
\frac{1}{(2\pi)^n}\int_{\mathbb R^n}\!\int_{\mathbb R^n}\!e^{-ix\cdot y}
\phi_k(\xi+y)_{\alpha,\beta}\alpha_x(\rho(\xi)_{\beta,\gamma})
\,\mathrm dx\,\mathrm dy,
\end{align*}
the rest of the proof reduces to the $r=1$ case, applying it to each summand
in the above sum.
\end{proof}
Let $\lambda(\xi)=(1+\xi_1^2+\cdots+\xi_n^2)^{1/2}\mathrm{id}_E$.
Consider the following inner product on $E$.
\begin{define}
Define the Sobolev inner product
$\langle\cdot,\cdot\rangle_s:E\times E\rightarrow\mathbb C$ by
$$\langle\vec{a},\vec{b}\rangle_s:=
\langle P_{\lambda^s}(\vec{a}),P_{\lambda^s}(\vec{b})\rangle
=\sum_{j,m}(1+|m_1|^2+\cdots+|m_n|^2)^s\overline{b_{j,m}}a_{j,m}.$$
\end{define}
Note that for $s=0$ this agrees with
$\langle\cdot,\cdot\rangle$.
This inner product induces the following norm.
\begin{define}
Define the Sobolev norm $||\cdot||_s:E\rightarrow\mathbb R_{\ge 0}$ by
$$||\vec{a}||_s^2:=\langle P_{\lambda^s}(\vec{a}),P_{\lambda^s}(\vec{a})\rangle
=\sum_{j,m}(1+|m_1|^2+\cdots+|m_n|^2)^s|a_{j,m}|^2.$$
\end{define}
Using this norm, we can define the analog of Sobolev spaces on $E$.
\begin{define}
Define the Sobolev space $H^s$ to be the completion of $E$
with respect to $||\cdot||_s$.
\end{define}
We can prove that a pseudo-differential operator of order $d\in\mathbb R$
continuously maps $H^s$ into $H^{s-d}$. However we must first prove the case
where $s=d$.
\begin{thm}\label{ws4}
Suppose $\rho$ is a matrix valued symbol of order $d$.
Then, for any $\vec{a}\in E$, $||P_{\rho}(\vec{a})||_0\le C||\vec{a}||_d$ for
some constant $C>0$ and $P_{\rho}$ defines a bounded operator
$P_{\rho}:H^d\rightarrow H^0$.
\end{thm}
\begin{proof}
Let $F:=\{f_j:1\le j\le r\}$ be an orthogonal eigenbasis of $e$ normalized
with respect to $\langle\cdot,\cdot\rangle$.
Note that $\{\prod_gU_g^{m_g}f_j:m\in\mathbb Z^n,1\le j\le r\}$
is an orthogonal basis of $E$
considered as a $\mathbb C$-vector space, with respect to
$\langle\cdot,\cdot\rangle_s$. We have
$||\prod_gU_g^{m_g}f_j||_0^2=1$,
$||\rho_m(\xi)_{h,j}\prod_gU_g^{m_g}f_j||_0^2=|\rho_m(\xi)_{h,j}|^2$,
and $||\rho(\xi)_{h,j}f_j||_0^2=\sum_m|\rho_m(\xi)_{h,j}|^2$.
Since $\rho_{h,j}$ is of order $d$, we have
$||\rho(\xi)_{h,j}||_0\le C_{\rho}(1+|\xi|)^d$, and
since $(1-|\xi|)^2\ge 0$ gives us $(1+|\xi|)^2\le 2(1+|\xi|^2)$, we have
$$||\rho(\xi)_{h,j}||_0^2
\le C_{\rho}^2(1+|\xi|)^{2d}\le C_{\rho}^22^d(1+|\xi|^2)^d.$$
Let $k_{\rho}:=C_{\rho}^22^d$. Then we have
$$\sum_m|\rho_m(\xi)_{h,j}|^2\le k_{\rho}(1+|\xi|^2)^d.$$
Let $e_{s,m}:=(1+|m_1|^2+\cdots+|m_n|^2)^{-s/2}\prod_gU_g^{m_g}$ and
$E_s:=\{e_{s,m}\mid n,m\in\mathbb Z\}$. By definition we have $E_sF$
orthonormal with respect to $\langle\cdot,\cdot\rangle_s$. It suffices to
prove this theorem for the case $\vec{a}=e_{d,m}f_j$
by the orthonormality of $E_dF$
since
$$||P_{\rho}(\vec{a})||_0^2=\sum_{j,m}|a_{j,m}|^2
||P_{\rho}(e_{d,m}f_j)||_0^2$$
and
$$||\vec{a}||_d^2=\sum_{j,m}|a_{j,m}|^2||e_{d,m}f_j||_d^2.$$
Since $||e_{d,m}f_j||_d^2=1$, it suffices to show that
$$||P_{\rho}(e_{d,m}f_j)||_0^2\le K$$ for some constant $K>0$.
We have
\begin{align*}
||P_{\rho}(e_{d,m}f_j)||_0^2
&= ||\rho(m)e_{d,m}f_j||_0^2 \\
&= ||\rho(m)(1+|m_1|^2+\cdots+|m_n|^2)^{-d/2}\prod_gU_g^{m_g}f_j||_0^2 \\
&= \left|\left|\sum_k
\rho_k(m)\prod_gU_g^{k_g}(1+|m_1|^2+\cdots+|m_n|^2)^{-d/2}\prod_gU_g^{m_g}f_j
\right|\right|_0^2 \\
&= (1+|m_1|^2+\cdots+|m_n|^2)^{-d}\left|\left|\sum_k
\rho_k(m)\prod_gU_g^{k_g}\prod_gU_g^{m_g}f_j\right|\right|_0^2 \\
&= (1+|m_1|^2+\cdots+|m_n|^2)^{-d}\left|\left|\sum_k
\rho_k(m)w(m,k)\prod_gU_g^{k_g+m_g}f_j
\right|\right|_0^2 \\
&= (1+|m_1|^2+\cdots+|m_n|^2)^{-d}\left|\left|\sum_k
\rho_{k-m}(m)w(m,k-m)\prod_gU_g^{k_g} f_j
\right|\right|_0^2 \\
&= (1+|m_1|^2+\cdots+|m_n|^2)^{-d}\left|\left|\sum_{h,k}
\rho_{k-m}(m)_{h,j}w(m,k-m)\prod_gU_g^{k_g} f_j
\right|\right|_0^2 \\
&= (1+|m_1|^2+\cdots+|m_n|^2)^{-d}
\sum_{h,k}|\rho_{k-m}(m)_{h,j}|^2 \\
&= (1+|m_1|^2+\cdots+|m_n|^2)^{-d}
\sum_{h,k}|\rho_k(m)_{h,j}|^2 \\
&\le (1+|m_1|^2+\cdots+|m_n|^2)^{-d}rk_{\rho}(1+|m_1|^2+\cdots+|m_n|^2)^d
=rk_{\rho}
\end{align*}
where $$w(k,m):=\prod_{j=1}^nU_j^{m_j}\prod_{j=1}^nU_j^{k_j}
\left(\prod_{j=1}^nU_j^{m_j+k_j}\right)\in S^1\subset\mathbb C$$
so our desired constant is $K=rk_{\rho}=rC_{\rho}^22^d$ and we are done.
\end{proof}
For the general case $s\ne d$ we need to prove a lemma saying that
$||\cdot||_s=||\cdot||_{s-t}\circ P_{\lambda^t}$.
\begin{lem}\label{swn4}
For any $\vec{a}\in E$ and $s,t\in\mathbb R$,
$\vec{a}\in H^s$ if and only if
$P_{\lambda^t}(\vec{a})\in H^{s-t}$
with $||\vec{a}||_s=||P_{\lambda^t}(\vec{a})||_{s-t}$.
\end{lem}
\begin{proof}
Suppose that $\vec{a}\in H^s$ or $P_{\lambda^t}(\vec{a})\in H^{s-t}$. Then
\begin{align*}
||P_{\lambda^t}(\vec{a})||_{s-t}^2
&= \sum_{j,m}(1+|m_1|^2+\cdots+|m_n|^2)^{s-t}\lambda^{2t}(m)|a_{j,m}|^2 \\
&= \sum_{j,m}(1+|m_1|^2+\cdots+|m_n|^2)^{s-t}(1+|m_1|^2+\cdots+|m_n|^2)^t
|a_{j,m}|^2 \\
&= \sum_{j,m}(1+|m_1|^2+\cdots+|m_n|^2)^s|a_{j,m}|^2 \\
&= ||\vec{a}||_s
\end{align*}
so we know that $\vec{a}\in H^s$ and $P_{\lambda^t}(\vec{a})\in H^{s-t}$.
\end{proof}
Then the general case follows quite easily.
\begin{cor}
Suppose $\rho$ is a matrix valued symbol of order $d$.
Then $||P_{\rho}(\vec{a})||_{s-d}\le C||\vec{a}||_s$ for
some constant $C>0$ and $P_{\rho}$ defines a bounded operator
$P_{\rho}:H^s\rightarrow H^{s-d}$.
\end{cor}
\begin{proof}
By Lemma \ref{swn4}, we have
$||P_{\rho}(\vec{a})||_{s-d}=||P_{\lambda^{s-d}}(P_{\rho}(\vec{a}))||_0$.
By Proposition \ref{mvs2}(b), the matrix valued symbol
$\sigma(P_{\lambda^{s-d}}\circ P_{\rho})$ is of order $d+(s-d)=s$,
so Theorem \ref{ws4} gives us
$||P_{\lambda^{s-d}}(P_{\rho}(\vec{a}))||_0\le C||\vec{a}||_s$
for some constant $C>0$.
\end{proof}
We can also define an analog of the $C^k$ norm on $E$.
\begin{define}
Define the $C^k$ norm
$||\cdot||_{\infty,k}:E\rightarrow\mathbb R_{\ge 0}$
as follows:
$$||\vec{a}||_{\infty,k}:=\sum_{|\ell|\le k}||\delta^{\ell}(\vec{a})||_{C^*}$$
where the $C^*$ norm $||\cdot||_{C^*}$ is given by
$$||\vec{a}||_{C^*}^2
:=\sup\{|\lambda|:\vec{a}^*\vec{a}-\lambda\cdot 1\text{ not invertible}\}.$$
\end{define}
Since, for arbitrary $\vec{a}=\sum_{j,m}a_{j,m}\prod_gU_g^{m_g}f_j$,
\begin{align*}
D_s^{\ell}\alpha_s(\vec{a})
&=(-i\partial_s)^{\ell}\sum_{j,m}e^{is\cdot m}a_{j,m}
\prod_{g=1}^nU_g^{m_g}f_j \\
&=\sum_{j,m}m^{\ell}e^{is\cdot m}a_{j,m}\prod_{g=1}^nU_g^{m_g}f_j \\
&=\sum_{j,m}\delta^{\ell}e^{is\cdot m}a_{j,m}\prod_{g=1}^nU_g^{m_g}f_j \\
&=\delta^{\ell}\alpha_s(\vec{a})
\end{align*}
we have $(A_{\theta}^k)^re=C^k$.

We can easily prove an analog of the Sobolev lemma on $E$ as follows.
\begin{thm}
For $s>k+1$, $H^s\subseteq(A_{\theta}^k)^re$.
\end{thm}
\begin{proof}
First consider the case $k=0$.
Note that $||\cdot||_{\infty,0}=||\cdot||_{C^*}$
so for arbitrary $a_{j,m}\prod_{g=1}^nU_g^{m_g}f_j$ we have
$$||a_{j,m}\prod_{g=1}^nU_g^{m_g}f_j||_{\infty,0}^2=\sup\{|\lambda|:
|a_{j,m}|^2-\lambda\cdot 1\text{ not invertible}\}=|a_{j,m}|^2$$
and for arbitrary $\vec{a}=\sum_{j,m}a_{j,m}\prod_{g=1}^nU_g^{m_g}f_j$ we have
$$||\vec{a}||_{\infty,0}^2
\le\sum_{j,m}||a_{j,m}\prod_{g=1}^nU_g^{m_g}f_j||_{\infty,0}^2
=\sum_{j,m}|a_{j,m}|^2
=||\vec{a}||_0^2$$ by the triangle inequality.
We have
$$\vec{a}
=\sum_{j,m}a_{j,m}\lambda^s(m)\lambda^{-s}(m)\prod_{g=1}^nU_g^{m_g}f_j$$
so by the Cauchy-Schwarz inequality we get
$$||\vec{a}||_0^2\le||\vec{a}||_s^2\sum_m(1+|m_1|^2+\cdots+|m_n|^2)^{-s}.$$
Since $2s>2$, $(1+|m_1|^2+\cdots+|m_n|^2)^{-s}$ is summable over
$m\in\mathbb Z^n$
and $j\in\{1,\ldots,r\}$ so $||\vec{a}||_0^2\le C||\vec{a}||_s$. Thus we get
$||\vec{a}||_{\infty,0}\le||\vec{a}||_0\le C||\vec{a}||_s$ and
$H^s\subseteq(A_{\theta}^0)^re$.

Now suppose $k>0$. Using what we've proven for the previous case, we have
\begin{align*}
||\delta^{\ell}(\vec{a})||_{\infty,0}
&\le C||\delta^{\ell}(\vec{a})||_{s-|\ell|} \\
&= C||\sum_{j,m}m^{\ell}a_{j,m}\prod_{g=1}^nU_g^{m_g}f_j||_{s-|\ell|} \\
&< C||\sum_{j,m}(1+|m_1|^2+\cdots+|m_n|^2)^{|\ell|}
a_{j,m}\prod_{g=1}^nU_g^{m_g}f_j||_{s-|\ell|} \\
&= C||P_{\lambda^{|\ell|}}(\vec{a})||_{s-|\ell|} \\
&= C||\vec{a}||_s
\end{align*}
for $|\ell|\le k$ since $s-|\ell|\ge s-k>1$.
Therefore,
$$||\vec{a}||_{\infty,k}=\sum_{|\ell|\le k}||\delta^{\ell}(\vec{a})||_{\infty,0}
\le\sum_{|\ell|\le k}C||\vec{a}||_s
\le C||\vec{a}||_s(k+1)(k+2)/2$$
and we get $H^s\subseteq(A_{\theta}^k)^re$.
\end{proof}
We get the following corollary.
\begin{cor}
$\displaystyle\bigcap_{s\in\mathbb R}H^s=(A_{\theta}^{\infty})^re$.
\end{cor}
\begin{proof}
Suppose $a\in\bigcap_{s\in\mathbb R}H^s$. Then for any
$k\in\mathbb Z_{\ge 0}$, $\vec{a}\in H^{k+2}$, so by the theorem we just proved,
$\vec{a}\in(A_{\theta}^k)^re$.
Consequently $\vec{a}\in(A_{\theta}^{\infty})^re$, so
$\bigcap_{s\in\mathbb R}H^s\subseteq(A_{\theta}^{\infty})^re$.

Suppose $a\in(A_{\theta}^{\infty})^re$. Then since
$H^s$ is the completion of $(A_{\theta}^{\infty})^re$
with respect to $||\cdot||_s$,
$(A_{\theta}^{\infty})^re\subseteq H^s$ for all $s\in\mathbb R$, and
$(A_{\theta}^{\infty})^re\subseteq\bigcap_{s\in\mathbb R}H^s$.
\end{proof}
We can also prove an analog of the Rellich lemma on $E$.
\begin{thm}
Let $\{\vec{a}_N\}\in(A_{\theta}^{\infty})^re$ be a sequence. Suppose that there
is a constant $C$ so that $||\vec{a}_N||_s\le C$ for all $N$. Let $s>t$. Then
there is a subsequence $\{\vec{a}_{N_j}\}$ that converges in $H^t$.
\end{thm}
\begin{proof}
Let $F:=\{f_j:1\le j\le r\}$ be a set of eigenvectors of $e$ normalized
with respect to $\langle\cdot,\cdot\rangle$, where $f_j$ has corresponding
eigenvalue $\lambda_j$ for $1\le j\le r$.
Let $e_{s,m}:=(1+|m_1|^2+\cdots+|m_n|^2)^{-s/2}\prod_gU_g^{m_g}$ and
$E_s:=\{e_{s,m}\mid n,m\in\mathbb Z\}$.
$E_sF$ is an orthonormal basis
with respect to $\langle\cdot,\cdot\rangle_s$, so we can write
$\vec{a}_N:=\sum_{h,k}a_{N,h,k}e_{s,k}f_h$. Then
$$|a_{N,h,k}|^2\le\sum_{h,k}|a_{N,h,k}|^2\le C^2$$
and $|a_{N,h,k}|\le C$.
Applying the Arzela-Ascoli theorem to $\{a_{N,h,k}\}$ for some fixed
$(h,k)$, we can get a subsequence
$\{a_{N_j,h,k}\}$ of $\{a_{N,h,k}\}$
such that for any $\epsilon>0$ there exists $M(\epsilon)\in\mathbb N$
such that $|a_{N_i,h,k}-a_{N_j,h,k}|<\epsilon$ whenever
$i,j\ge M(\epsilon)$.
Do this for all $1\le h\le r$ and $|k_1|^2+\cdots+|k_n|^2\le R$,
replacing $\{a_N\}$ with $\{a_{N_j}\}$
each time. Then we get a subsequence $\{a_{N_j}\}$ of $\{a_N\}$
such that for any $\epsilon>0$ there exists $M(\epsilon)\in\mathbb N$
such that, for all $1\le h\le r$ and $|k_1|^2+\cdots+|k_n|^2\le R$,
$|a_{N_i,h,k}-a_{N_j,h,k}|<\epsilon$ whenever $i,j\ge M(\epsilon)$.
Now consider the sum $$||a_{N_i}-a_{N_j}||_t^2
=\sum_{h,k}|a_{N_i,h,k}-a_{N_j,h,k}|^2(1+|k_1|^2+\cdots+|k_n|^2)^{t-s}.$$
Decompose it into two parts: one where $|k_1|^2+\cdots+|k_n|^2>R^2$ and one
where $|k_1|^2+\cdots+|k_n|^2\le R^2$.
On $|k_1|^2+\cdots+|k_n|^2>R^2$ we estimate
$$(1+|k_1|^2+\cdots+|k_n|^2)^{t-s}<(1+R^2)^{t-s}$$ so that
\begin{align*}
\sum_h\sum_{|k_1|^2+\cdots+|k_n|^2\ge R^2}|a_{N_i,h,k}-a_{N_j,h,k}|^2
(1+|k_1|^2+\cdots+|k_n|^2)^{t-s}
&<(1+R^2)^{t-s}\sum_{h,k}|a_{N_i,h,k}-a_{N_j,h,k}|^2 \\
&\le2rC^2(1+R^2)^{t-s}.
\end{align*}
If $\epsilon>0$ is given, we choose $R$ so that $2rC^2(1+R^2)^{t-s}<\epsilon$.
The remaining part of the sum is over $|k_1|^2+\cdots+|k_n|^2\le R^2$ and
can be bounded above by
$\epsilon':=\epsilon-2rC^2(1+R^2)^{t-s}$ if
$i,j\ge M(\sqrt[n]{\epsilon'}/(2R+1))$
because a ball of radius $R$ centered at the origin is contained in a cube
of side length $2R$ that has $(2R+1)^2$ lattice points.
Then the total sum is bounded above by $\epsilon$, and we are done.
\end{proof}
\section{Acknowledgement}
I would like to thank Farzad Fathizadeh for suggesting the problem of
filling in the details in the pseudodifferential calculus on the noncommutative
$n$ torus. I would also like thank Matilde Marcolli for helping me make plans
to take my candidacy exam. I would like to thank Vlad Markovic and Eric Rains
for agreeing to be on my candidacy exam committee.

\bibliography{paper}{}
\bibliographystyle{plain}
\end{document}